\documentclass[11pt]{amsart}
\usepackage{amssymb,amsmath}
\usepackage{tikz}

\newtheorem{theorem}{Theorem}[section]
\newtheorem{proposition}[theorem]{Proposition}
\newtheorem{lemma}[theorem]{Lemma}

\newtheorem{corollary}[theorem]{Corollary}

\newtheorem{conjecture}[theorem]{Conjecture}
\numberwithin{equation}{section}

\newcommand{\av}{{\mathbf a}}
\newcommand{\bv}{{\mathbf b}}
\newcommand{\cv}{{\mathbf c}}
\newcommand{\dv}{{\mathbf d}}
\newcommand{\tv}{{\mathbf t}}
\newcommand{\ab}{\av\bv}
\newcommand{\abt}{\av\bv\tv}
\newcommand{\cd}{\cv\dv}

\newcommand{\zabs}{\Zzz\langle\langle\av,\bv\rangle\rangle}
\newcommand{\zcds}{\Zzz\langle\langle\cv,\dv\rangle\rangle}
\newcommand{\zabts}{\Zzz\langle\langle\av,\bv,\tv\rangle\rangle}
\newcommand{\zcdts}{\Zzz\langle\langle\cv,\dv,\tv\rangle\rangle}
\newcommand{\hz}{\widehat{0}}
\newcommand{\ho}{\widehat{1}}
\newcommand{\antiflip}{\rotatebox[origin=c]{90}{\small $\forall$}}

\newcommand{\Qqq}{{\mathbb Q}}

\newcommand{\Zzz}{{\mathbb Z}}

\DeclareMathOperator{\Bin}{Bin}
\newcommand{\VerticalLine}{\hspace*{-\arraycolsep}\vline\hspace*{-\arraycolsep}}
\newcommand{\BMTwoByTwo}[4]
{\left(\begin{array}{c c c}
#1 & \VerticalLine & #2 \\ \hline
#3 & \VerticalLine & #4
\end{array}\right)}

\begin{document}

\vspace*{-20 mm}

\title{Two classes of level Eulerian posets}
\author{Richard Ehrenborg}

\address{Department of Mathematics, University of Kentucky,
Lexington, KY 40506-0027.
\hfill\break \tt https://www.math.uky.edu/\~{}jrge/,
richard.ehrenborg@uky.edu.}



\begin{abstract}
We present two classes of level Eulerian posets.
Both classes contain intervals of rank $k+1$ whose $\cd$-index
is the sum over all $\cd$-monomials $w$ of degree $k$
and the coefficient of the monomial~$w$ is $r$ to the power of the number of $\dv$'s in~$w$.
We also show that the order complexes of every interval in the first
class are homeomorphic to spheres.
\end{abstract}

\maketitle

\section{Introduction}

Level posets were introduced in~\cite{Ehrenborg_Hetyei_Readdy}
as posets where the cover relations between
two adjacent ranks are the same independent of the rank.
The cover relations are naturally encoded by a $0,1$-matrix~$M = (m_{i,j})$,
that is, the element $(i,s)$ is covered by $(j,s+1)$ if and only if $m_{i,j} = 1$.
Such posets are naturally infinite and ranked.
See Figure~\ref{figure_M_0_1_2} for three examples.

More interestingly, the paper
constructs such a poset which is Eulerian;
see~\cite[Example~4.7]{Ehrenborg_Hetyei_Readdy}.
This level poset contains intervals whose $\cd$-index
have all the $\cd$-coefficients equal to~$1$;
see~\cite[Example~8.2]{Ehrenborg_Hetyei_Readdy}.
This example is the second poset displayed in
Figure~\ref{figure_M_0_1_2}.

Recall that the $\ab$-index of a poset $P$ is
an encoding of the flag $f$-vector of the poset as a polynomial $\Psi(P)$
in the non-commutative variables $\av$ and $\bv$.
Furthermore, when the poset is Eulerian
the $\ab$-index can be expressed in terms of
$\cv = \av+\bv$ and $\dv = \av\bv+\bv\av$;
see~\cite{Bayer_Klapper,Stanley}.
For a level posets a more theoretical result is proven in~\cite{Ehrenborg_Hetyei_Readdy},
namely, that the sum
\begin{align}
\Psi_{i,j} = \sum_{k \geq 1} \Psi([(i,0), (j,k)])
\label{equation_Psi_series}
\end{align}
is a non-commutative rational series in the variables $\av$ and $\bv$.
Furthermore, for level Eulerian posets, this series
is a non-commutative rational series in the variables $\cv$ and $\dv$.

In this paper we present two new classes of level Eulerian posets.
Furthermore, we calculate their $\cd$-series.
Interestingly, both classes contain intervals
where the $\cd$-coefficients are $r$ to the power of the number
of $\dv$'s in the associated $\cd$-monomial.
This phenomenon is reflected in the fact that the denominator of the
rational $\cd$-index series is the polynomial $1 - \cv - r \cdot \dv$.

For the first class of level Eulerian posets we also show that
all of the intervals are shellable. As a consequence the order complex
of each interval in this poset is homeomorphic to a sphere.

We end the paper with concluding remarks and open questions.

\section{Preliminaries}

For a non-negative matrix $M = (m_{i,j})_{1 \leq i \leq m, 1 \leq j \leq n}$
let $\Bin(M)$ be the $0,1$-matrix defined by
\begin{align*}
\Bin(M)_{i,j} 
=
\begin{cases}
1 & m_{i,j} > 0, \\ 0 & m_{i,j} = 0.
\end{cases}
\end{align*}
Furthermore, we let $J$ be the matrix with all the entries equal to~$1$.
Note that the size of the matrix~$J$ will be understood from the context.
We call a non-negative square matrix $M$ {\em primitive} if
there is a positive integer $k$ such that $\Bin(M^{k}) = J$.
The smallest such power $\gamma$ is called the {\em exponent}.

An equivalent way to encode the information of the matrix~$M$
is the associated digraph, where there is a directed edge
$i \longrightarrow j$ if and only if $m_{i,j} = 1$.
Hence the matrix~$M$ is primitive
if and only if there exist a positive integer~$k$ such that 
there is a directed path from any vertex to any other vertex with exactly $k$~steps.

A poset $P$ is {\em ranked} if there is a function $\rho : P \longrightarrow \Zzz$
such that for all cover relations $x \prec y$ we have $\rho(x) + 1 = \rho(y)$.
For two elements $x$ and $y$ in a poset such that $x \leq y$, the {\em interval}
$[x,y]$ is the set $\{z \in P : x \leq z \leq y\}$.
A poset $P$ is {\em Eulerian} if all non-singleton intervals have the same number
of elements of even rank as the number elements of odd rank.
To verify the Eulerian condition, it is enough to do so for intervals
of even rank; see~\cite[Lemma~4.4]{Ehrenborg_k-Eulerian}
or~\cite[Lemma~2.6]{Ehrenborg_Readdy_Eulerian_binomial}.
We call the condition that all the intervals of length~$k$ are Eulerian
the {\em rank~$k$ Eulerian condition}.

For an $n \times n$ $0,1$-matrix $M$ the associated
{\em level poset} is a poset on the set $\{1,2, \ldots, n\} \times \Zzz$
such that the cover relation is given by
$(i,s) \prec (j,s+1)$ if and only if $m_{i,j} = 1$.
Note that the order relation $(i,s) \leq (j,p)$ holds if and only if
the equality $\Bin(M^{p-s})_{i,j} = 1$ holds.
This equivalence also holds in the extreme case when $p=s$
since $M^{0}$ is the $n \times n$ identity matrix.
Hence the number of elements of rank $s$ in the interval
$[(i,0),(j,p)]$ is the $(i,j)$ entry of the matrix product
$\Bin(M^{s}) \cdot \Bin(M^{p-s})$.
The rank $p$ Eulerian condition is therefore the following matrix identity
\begin{align*}
\sum_{s=0}^{p} (-1)^{s} \cdot \Bin(M^{s}) \cdot \Bin(M^{p-s})    =    0.
\end{align*}
By the comment above, it is enough to verify this condition for even positive integers $p$.
Note that an equivalent formulation of the rank $2$ Eulerian condition
is the following identity:
\begin{align*}
M^{2} = 2 \cdot \Bin(M^{2}) .
\end{align*}
This condition is equivalent to that every rank~$2$ interval is isomorphic to a diamond,
that is, the rank~$2$ Boolean algebra.

\begin{proposition}
Assume that $M$ is an $n \times n$ primitive $0,1$-matrix with exponent $\gamma$.
\begin{itemize}
\item[(a)]
Let $W$ be the matrix 
\begin{align*}
W = \sum_{s=1}^{\gamma-1} (-1)^{s} \cdot \Bin(M^{s}) .
\end{align*}
If all of the row sums and the column sums of $W$ is
$(-1)^{\gamma-1} \cdot n/2 - 1$
then the matrix $M$ satisfies the rank $2k$ Eulerian condition for $k \geq \gamma$.
\item[(b)]
If the matrix $M$ satisfies the Eulerian condition of rank $2\gamma$
and $(J - \Bin(M^{\gamma-1}))^{2} = 0$ then
the matrix $M$ satisfies the Eulerian condition of rank $2\gamma - 2$.
\end{itemize}
\label{proposition_last_Eulerian_condition}
\end{proposition}
\begin{proof}
For $k \geq \gamma$
the sum in the Eulerian condition of rank $2k$ can be rewritten as
\begin{align}
\sum_{s=0}^{2k}
(-1)^{s} \cdot \Bin(M^{s}) \cdot \Bin(M^{2k-s}) 
& =
2J
+
W \cdot J
+
J \cdot W
+
(-1)^{\gamma} \cdot J^{2} .
\label{equation_rank_2k}
\end{align}
This identity follows from the facts
the two terms $s=0$ and $s=2k$ both yield the matrix $J$,
the terms $1 \leq s \leq \gamma-1$ yield $W \cdot J$,
the terms $2k-\gamma+1 \leq s \leq 2k-1$ yield $J \cdot W$
and the sum of $2k-2\gamma+1$ middle terms yield $(-1)^{\gamma} \cdot J^{2}$.
The matrix on the right-hand side of~\eqref{equation_rank_2k} equals
$(2 + 2c + (-1)^{\gamma} \cdot n) \cdot J = 0$
where $c$ is the row and column sums of~$W$.
This proves statement~(a).

To prove statement~(b),
note the sum in the rank $2\gamma-2$ Eulerian condition can be rewritten as
\begin{align}
\label{equation_rank_2_gamma-2}
&
(-1)^{\gamma-1} \cdot \Bin(M^{\gamma-1})^{2}
+
\sum_{\substack{s=0 \\ s \neq \gamma-1}}^{2\gamma-2}
(-1)^{s} \cdot \Bin(M^{s}) \cdot \Bin(M^{2\gamma-2-s}) \\
& =
\nonumber
(-1)^{\gamma-1} \cdot \Bin(M^{\gamma-1})^{2}
+
\sum_{\substack{s=0 \\ s \neq \gamma-1,\gamma,\gamma+1}}^{2\gamma}
(-1)^{s} \cdot \Bin(M^{s}) \cdot \Bin(M^{2\gamma-s}) ,
\end{align}
where we use that $\Bin(M^{s}) = J = \Bin(M^{s+2})$ for $s \geq \gamma$.
Since $J = \Bin(M^{\gamma}) = \Bin(M^{\gamma+1})$,
we reformulate the condition that the square vanishes as follows:
\begin{align}
\label{equation_square}
0 
& =
(-1)^{\gamma} \cdot (\Bin(M^{\gamma-1}) - J)^{2} \\
\nonumber
& =
(-1)^{\gamma} \cdot \Bin(M^{\gamma-1})^{2}
+
\sum_{s = \gamma-1}^{\gamma+1}
(-1)^{s} \cdot \Bin(M^{s}) \cdot \Bin(M^{2\gamma-s}) .
\end{align}
Adding these two identities~\eqref{equation_rank_2_gamma-2}
and~\eqref{equation_square}
yields the sum of the rank $2\gamma$ Eulerian condition
which is equal to $0$. Hence the rank $2\gamma-2$ Eulerian condition holds.
\end{proof}

\section{The $\ab$- and $\cd$-index and coalgebra techniques}

A poset $P$ is {\em graded} if it is ranked and has a unique minimal element $\hz$
and a unique maximal element $\ho$. We say that $P$ has rank $n+1$ if
$\rho(\hz) = 0$ and $\rho(\ho) = n+1$.

Let $P$ be a graded poset of rank $n+1$.
Let $S = \{s_{1} < s_{2} < \cdots < s_{k}\} \subseteq \{1,2,\ldots, n\}$.
We define the flag $f$-vector entry $f_{S}(P) = f_{S}$ to be the number of chains~$c$
in the poset $P$ with elements having ranks in the set $S$, that is,
\begin{align*}
f_{S}
& =
\{\{\hz = x_{0} < x_{1} < \cdots < x_{k} < x_{k+1} = \ho\} \subseteq P : 
    \rho(x_{i}) = s_{i} \text{ for } i =0,1, \ldots, k+1\} ,
\end{align*}
where $s_{0} = 0$ and $s_{k+1} = n+1$.
For two elements $x$ and $y$ in the poset $P$ such that $x \leq y$,
let the rank difference be denoted by $\rho(x,y) = \rho(y)-\rho(x)$.
The {\em $\ab$-index} is defined as
\begin{align*}
\Psi(P)
=
\sum_{c}
(\av-\bv)^{\rho(x_{0},x_{1})-1} \cdot \bv \cdot
(\av-\bv)^{\rho(x_{1},x_{2})-1} \cdot \bv \cdots
\bv \cdot (\av-\bv)^{\rho(x_{k-1},x_{k})-1} \cdot
\bv \cdot (\av-\bv)^{\rho(x_{k},x_{k+1})-1} ,
\end{align*}
where the sum is over all chains
$c = \{\hz = x_{0} < x_{1} < \cdots < x_{k} < x_{k+1} = \ho\}$
in the poset $P$.
Note that the $\ab$-index of a graded poset $P$ of rank $n+1$
is a homogeneous polynomial of degree $n$.

When the poset $P$ is Eulerian, the $\ab$-index $\Psi(P)$
can be expressed in terms of the variables
$\cv = \av+\bv$ and $\dv = \av\bv+\bv\av$.
This fact was first proved by 
Bayer and Klapper~\cite{Bayer_Klapper} for face lattices
of convex polytopes and
they extended it to Eulerian posets;
see~\cite[Theorem~4]{Bayer_Klapper}.
A direct proof for Eulerian posets
was given by Stanley~\cite[Theorem~1.1]{Stanley}.

We now introduce the coalgebra techniques of the paper~\cite{Ehrenborg_Readdy_coproducts}.
However, in order to work in the setting of formal series, we use the
reformulation of~\cite[Section~8]{Ehrenborg_Hetyei_Readdy},
that is, we introduce a variable $\tv$ to play the role of the tensor sign.
Let $\zabs$ and $\zabts$ be the two rings of formal series
in the non-commutative variables $\av$ and $\bv$,
respectively, $\av$, $\bv$ and $\tv$.
Let $\Delta: \zabs \longrightarrow \zabts$ be a derivation such that
$\Delta(1) = 0$ and $\Delta(\av) = \Delta(\bv) = \tv$.
Note that
$\Delta(\cv) = 2 \tv$ and $\Delta(\dv) = \cv\tv+\tv\cv$.
Thus the derivation restricts to a linear map $\zcds \longrightarrow \zcdts$.
Recall that
for an $\ab$-series $u$, where $u$ has no constant term,
we have the identity $1/(1-u) = \sum_{k \geq 0} u^{k}$.
Hence we obtain
\begin{align}
\Delta\left(\frac{1}{1-u}\right) = \frac{1}{1-u} \cdot \Delta(u) \cdot \frac{1}{1-u} .
\label{equation_Delta_on_geometric_series}
\end{align}
For a $0,1$-matrix $M$ introduce the series $K_{M}(t)$ by
\begin{align}
K_{M}(t) = \sum_{k \geq 0} \Bin(M^{k+1}) \cdot t^{k} .
\label{equation_K}
\end{align}
Let $\Psi$ be the matrix of the $\ab$-series $\Psi_{i,j}$ from
equation~\eqref{equation_Psi_series}, that is,
\begin{align*}
\Psi = (\Psi_{i,j})_{1 \leq i,j \leq n} .
\end{align*}
Theorem~8.1 in~\cite{Ehrenborg_Hetyei_Readdy} gives a method
to verify the $\ab$-series for a level poset:
\begin{theorem}[Ehrenborg--Hetyei--Readdy]
The $\ab$-series matrix $\Psi_{M}$ of a level poset associated to the $0,1$-matrix $M$
is the unique solution to the equation system
\begin{align*}
\Psi_{M}\vrule_{\: \av=t,\bv=0}
& =
K_{M}(t) , \\
\Delta(\Psi_{M})
& =
\Psi_{M} \cdot \tv \cdot \Psi_{M} .
\end{align*}
\label{theorem_unique_solution}
\end{theorem}
Note that when $\Psi_{M}$ is expressed in terms $\cv$'s and $\dv$'s,
we have the substitution
\begin{align*}
\Psi_{M}\vrule_{\: \av=t,\bv=0} = \Psi_{M}\vrule_{\: \cv=t,\dv=0} .
\end{align*}

For any non-commutative monomial or word $u = u_{1} u_{2} \cdots u_{k}$
the reverse of $u$ is defined to be
$u^{*} = u_{k} \cdots u_{2} u_{1}$. 
For $\abt$-monomials
we extend the reversing map $u \longmapsto u^{*}$ to
$\zabts$ by linearity. Note that this map on $\zabts$ is an anti-isomorphism, that is,
$(u \cdot v)^{*} = v^{*} \cdot u^{*}$. Also observe $\cv^{*} = \cv$ and $\dv^{*} = \dv$
and hence this map also reverses $\cd$-monomials and thus also $\cd$-series.

Recall that the transpose of a matrix is the operation of reflecting
the matrix entries over the main diagonal.
We now introduce the operation of reflecting the matrix over its anti-diagonal.
Furthermore since some of the matrices and vectors we work with contain
non-commutative variables,
we will also include reversing these elements.
More formally, if $M = (m_{i,j})_{1 \leq i \leq m, 1 \leq j \leq n}$ is an $m \times n$ matrix
then we define $M^{\antiflip}$ to be the $n \times m$ matrix
such that
$(M^{\antiflip})_{i,j} = m_{m+1-j,n+1-i}^{*}$.
We use the flipped letter A to denote this operation:
the letter A for anti-diagonal and this letter is flipped over the antidiagonal.
Note that under this operation the entry $w$ in the upper right-hand corner
remains in the same location under this operation and maps to~$w^{*}$.
Similarly, the entry in the lower-left hand corner
also remains in the same location.
We also have the following two relations
\begin{align}
(M + N)^{\antiflip} & = M^{\antiflip} + N^{\antiflip} ,
&
(M \cdot N)^{\antiflip} & = N^{\antiflip} \cdot M^{\antiflip} .
\label{equation_anti_flip}
\end{align}
Note that the last relation would not hold if we did not
include the anti-isomorphism $w \longmapsto w^{*}$ in the definition.
Finally, note that for a vector $\vec{v}$ the two vectors
$\vec{v}^{T}$ and $\vec{v}^{\antiflip}$ are reverses of each other.

\section{The first class of level Eulerian posets}
\label{section_The_first_class}

We now present the first class of level Eulerian posets
and compute its $\cd$-series.
Let $r$ be a nonnegative integer.
To construct the $(2r+2) \times (2r+2)$-matrix $M$
we define three $(r+1) \times (r+1)$-matrices:
\begin{itemize}
\item[$(A)$]
Let $A = (a_{i,j})_{0 \leq i,j \leq r}$
where
$a_{i,j} = 1$ if $i=0$
and $a_{i,j}=0$ otherwise.

\item[$(B)$]
Let $B = (b_{i,j})_{0 \leq i,j \leq r}$
where
$b_{i,j} = 1$ if
$0 \leq i-j \leq 1$
and $b_{i,j}=0$ otherwise.
Note that $B^{\antiflip} = B$.

\item[$(F)$]
Let $F = (f_{i,j})_{0 \leq i,j \leq r}$
where
$f_{i,j} = 1$ if $i=0$ or $j=r$,
and $f_{i,j}=0$ otherwise.
Again, $F^{\antiflip} = F$.
\end{itemize}
We begin with a few relations among these matrices.
\begin{lemma}
The following seven relations hold among the four matrices $A$, $B$, $F$ and $J$:
\begin{align}
\label{equation_lemma_M_A}
A^{2} & = A, & {A^{\antiflip}}^{2} & = A^{\antiflip} , \\
\label{equation_lemma_M_J}
B \cdot J & = 2 \cdot J - A, &
J \cdot B & = 2 \cdot J - A^{\antiflip}, &
J \cdot A & = A^{\antiflip} \cdot J = J, \\
\label{equation_lemma_M_F}
&& A \cdot B + B \cdot A^{\antiflip} & = 2 \cdot F.
\end{align}
\label{lemma_relations_M}
\end{lemma}
\begin{proof}
The two matrices $A$ and $A^{\antiflip}$ are projections
yielding~\eqref{equation_lemma_M_A}.
The three relations in~\eqref{equation_lemma_M_J}
follows since multiplication with the matrix $J$ is straightforward.
Note that $A \cdot B$ is the matrix
where the top row is $(2, 2, \ldots, 2, 1)$
and the remaining rows are zero.
Similarly,
$B \cdot A^{\antiflip}$ is the matrix where the last column is
$(1, 2, 2, \ldots, 2)^{T}$
and the remaining columns are zero.
Adding these two last observations yields
relation~\eqref{equation_lemma_M_F}.
\end{proof}

\begin{figure}
\begin{center}
\begin{tikzpicture}
\foreach \x in {0,1} {\foreach \y in {0,1,2,3,4,5}
{\draw[fill] (\x,\y) circle (0.04);};};

\foreach \y in {0,1,2,3,4}
{
\draw[-,very thick] (1,{\y}) -- (1,{\y+1}) -- (0,{\y}) -- (0,{\y+1});
\draw[-,very thin] (0,{\y+1}) -- (1,{\y});
};
\end{tikzpicture}
\hspace*{15mm}
\begin{tikzpicture}
\foreach \x in {0,1,2,3} {\foreach \y in {0,1,2,3,4,5}
{\draw[fill] (\x,\y) circle (0.04);};};

\foreach \y in {0,1,2,3,4}
{
\draw[-,very thin] (2,{\y}) -- (0,{\y+1}) -- (3,{\y}) -- (1,{\y+1}) -- cycle;
\draw[-,very thick] (0,{\y}) -- (0,{\y+1});
\draw[-,very thick] (0,{\y}) -- (1,{\y+1});
\draw[-,very thick] (3,{\y+1}) -- (2,{\y});
\draw[-,very thick] (3,{\y+1}) -- (3,{\y});
\draw[-,very thick] (0,{\y}) -- (2,{\y+1});
\draw[-,very thick] (2,{\y+1}) -- (1,{\y});
\draw[-,very thick] (1,{\y}) -- (3,{\y+1});
};
\end{tikzpicture}
\hspace*{15mm}
\begin{tikzpicture}
\foreach \x in {0,1,2,3,4,5} {\foreach \y in {0,1,2,3,4,5}
{\draw[fill] (\x,\y) circle (0.04);};};

\foreach \y in {0,1,2,3,4}
{
\foreach \a in {3,4,5} {\foreach \b in {0,1,2} {\draw[-,very thin] (\a,{\y}) -- (\b,{\y+1});};};
\foreach \b in {0,1,2} {\draw[-,very thick] (0,{\y}) -- (\b,{\y+1});};
\foreach \a in {3,4,5} {\draw[-,very thick] (\a,{\y}) -- (5,{\y+1});};
\draw[-,very thick] (0,{\y}) -- (3,{\y+1}) -- (1,{\y}) -- (4,{\y+1}) -- (2,{\y}) -- (5,{\y+1});
};
\end{tikzpicture}
\end{center}
\caption{The level Eulerian posets associated to matrix $M$ for $r=0$, $1$ and $2$.
The cover relations corresponding to the southwest block $J$ has been drawn
with thin lines.}
\label{figure_M_0_1_2}
\end{figure}
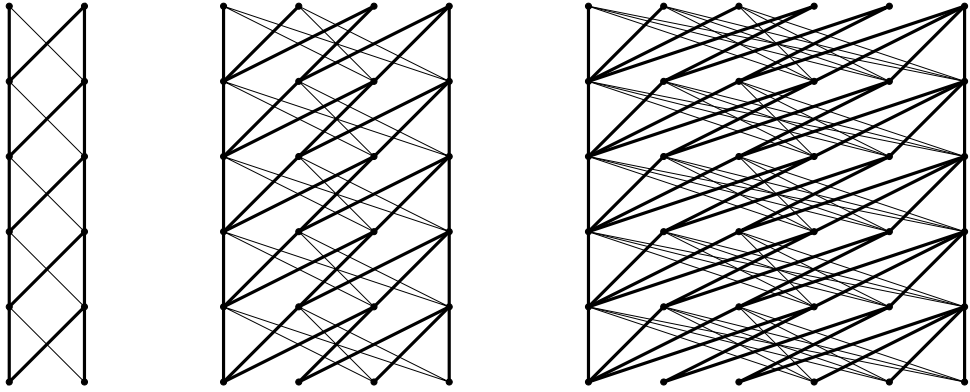

Let $M$ be the block matrix
\begin{align*}
M
=
\BMTwoByTwo{A}{B}{J}{A^{\antiflip}} .
\end{align*}
Note that $r=1$ yields Example~4.7 in~\cite{Ehrenborg_Hetyei_Readdy}
and $r=0$ yields
$M = \left(\begin{array}{c c}1&1\\1&1\end{array}\right)$,
that is, the infinite butterfly poset.
See Figure~\ref{figure_M_0_1_2}
for the level posets for $r = 0,1$ and $2$.
For $r=3$ the matrix $M$ and the matrix $\Bin(M^{2})$ are as follows
\begin{align*}
M
& =
\left(\begin{array}{c c c c c c c c c}
1 & 1 & 1 & 1 & \VerticalLine & 1 & 0 & 0 & 0 \\
0 & 0 & 0 & 0 & \VerticalLine & 1 & 1 & 0 & 0 \\
0 & 0 & 0 & 0 & \VerticalLine & 0 & 1 & 1 & 0 \\
0 & 0 & 0 & 0 & \VerticalLine & 0 & 0 & 1 & 1 \\ \hline
1 & 1 & 1 & 1 & \VerticalLine & 0 & 0 & 0 & 1 \\
1 & 1 & 1 & 1 & \VerticalLine & 0 & 0 & 0 & 1 \\
1 & 1 & 1 & 1 & \VerticalLine & 0 & 0 & 0 & 1 \\
1 & 1 & 1 & 1 & \VerticalLine & 0 & 0 & 0 & 1
\end{array}\right) ,
&
\Bin(M^{2})
& =
\left(\begin{array}{c c c c c c c c c}
1 & 1 & 1 & 1 & \VerticalLine & 1 & 1 & 1 & 1 \\
1 & 1 & 1 & 1 & \VerticalLine & 0 & 0 & 0 & 1 \\
1 & 1 & 1 & 1 & \VerticalLine & 0 & 0 & 0 & 1 \\
1 & 1 & 1 & 1 & \VerticalLine & 0 & 0 & 0 & 1 \\ \hline
1 & 1 & 1 & 1 & \VerticalLine & 1 & 1 & 1 & 1 \\
1 & 1 & 1 & 1 & \VerticalLine & 1 & 1 & 1 & 1 \\
1 & 1 & 1 & 1 & \VerticalLine & 1 & 1 & 1 & 1 \\
1 & 1 & 1 & 1 & \VerticalLine & 1 & 1 & 1 & 1
\end{array}\right) .
\end{align*}

\begin{proposition}
For $r \geq 1$ the matrix $M$ is primitive with exponent~$3$ and
$M$ is the adjacency matrix of a level Eulerian poset.
\label{proposition_M}
\end{proposition}
\begin{proof}
By Lemma~\ref{lemma_relations_M} and
that $J$ and $F$ are $0,1$-matrices,
the square of the matrix $M$ is given by
\begin{align*}
M^{2}
=
\BMTwoByTwo{A^{2} + BJ}{AB+BA^{\antiflip}}{JA+A^{\antiflip}J}{JB + {A^{\antiflip}}^{2}}
=
2 \cdot
\BMTwoByTwo{J}{F}{J}{J}
=
2 \cdot
\Bin(M^{2}) .
\end{align*}
This identity settles the Eulerian condition of rank $2$.
It is straightforward to observe that $\Bin(M^{3}) = J$
and hence $M$ is primitive with exponent $3$.

Note that all of the row and column sums of $- M + \Bin(M^{2})$ is $r = n/2-1$.
Hence by Proposition~\ref{proposition_last_Eulerian_condition} part (a)
we know that the Eulerian condition holds for even ranks
greater than or equal to $6$.
Finally, the rank $4$ Eulerian condition follows from
Proposition~\ref{proposition_last_Eulerian_condition} part~(b)
and the observation
\begin{align*}
\left(J - \Bin(M^{2})\right)^{2}
& =
\BMTwoByTwo{0}{J-F}{0}{0}^{2}
=
0.
\qedhere
\end{align*}
\end{proof}

We now will determine the $\cd$-index for all of the intervals in the level Eulerian poset associated
with the matrix~$M$.
Using that the exponent of the matrix $M$ is $3$,
we have that the power series $K_{M}(t)$ is given by
\begin{align*}
K_{M}(t) = \sum_{k \geq 0} \Bin(M^{k+1}) \cdot t^{k} = M + \Bin(M^{2}) \cdot t + J \cdot \frac{t^{2}}{1-t} .
\end{align*}
Next let $\phi$ denote the formal power series
\begin{align}
\label{equation_phi}
\phi 
& =
\frac{1}{1 - \cv - r\dv} 
=
\sum_{w} r^{\text{number of $\dv$s in $w$}} \cdot w ,
\end{align}
where the sum is over all $\cd$-monomials $w$.

Define the column vector $\vec{p}$ by
${\vec{p}\:}^{T} = (1, \underbrace{\cv, \ldots, \cv}_{r}, \underbrace{1, \ldots, 1}_{r+1})$.
Note that
${\vec{p}\:}^{\antiflip} = (\underbrace{1, \ldots, 1}_{r+1}, \underbrace{\cv, \ldots, \cv}_{r}, 1)$.
Also define the block $H$ by:
\begin{itemize}
\item[$(H)$]
Let $H$ be the $(r+1) \times (r+1)$ block matrix $B$ with $1$ subtracted off the upper right entry.
That is, for $r=0$ the matrix $H$ is the zero matrix $(0)$.
For $r \geq 1$ we have
$H = (h_{i,j})_{0 \leq i,j \leq r}$
where 
$h_{i,j} = 1$ if $0 \leq i-j \leq 1$,
$h_{0,r} = -1$
and
$h_{i,j} = 0$ otherwise.
Note that $H^{\antiflip} = H$.
\end{itemize}
Finally, let $C = C^{\antiflip}$ be the block matrix
$\BMTwoByTwo{0}{H}{0}{0}$.
\begin{proposition}
The $\cd$-index series for the level Eulerian poset associated to the matrix $M$ is given by
\begin{align*}
\Psi_{M}
& =
\vec{p} \cdot \phi \cdot {\vec{p}\:}^{\antiflip} + C .
\end{align*}
\end{proposition}
\begin{proof}
The first condition of Theorem~\ref{theorem_unique_solution} is straightforward to check
\begin{align*}
\Psi_{M}\vrule_{\: \cv=t,\dv=0}
& =
{\vec{p}\:\vrule_{\: \cv=t}} \cdot \frac{1}{1-t} \cdot {{\vec{p}\:}^{\antiflip}\:\vrule_{\: \cv=t}}
+
C
=
\left(\begin{array}{c c c c c}
J & \VerticalLine & tJ & \VerticalLine & J \\ \hline
tJ & \VerticalLine & t^{2}J & \VerticalLine & tJ \\ \hline
J & \VerticalLine & tJ & \VerticalLine & J
\end{array}\right)
\cdot \frac{1}{1-t}
+
C
=
K_{M}(t) ,
\end{align*}
where the $3 \times 3$ block matrix has the row partition $1$, $r$ and $r+1$
and 
the column partition is the reversed partition $r+1$, $r$ and~$1$.
Before verifying the second condition of Theorem~\ref{theorem_unique_solution},
observe 
\begin{align*}
{\vec{p}\:}^{\antiflip} \cdot \tv \cdot \vec{p} & = 2 \tv + r \cdot (\tv \cv + \cv \tv) = \Delta(\cv + r \cdot \dv) , \\
C \cdot \tv \cdot \vec{p} & = \Delta(\vec{p}) , \\
{\vec{p}\:}^{\antiflip} \cdot \tv \cdot C & = \Delta({\vec{p}\:}^{\antiflip}) , \\
C \cdot \tv \cdot C & = 0 = \Delta(C) .
\end{align*}
Next,  using the distributive law and the four identities above, we obtain:
\begin{align*}
\Psi_{M} \cdot \tv \cdot \Psi_{M}
& =
\left( \vec{p} \cdot \phi \cdot {\vec{p}\:}^{\antiflip} + C \right)
\cdot \tv \cdot
\left( \vec{p} \cdot \phi \cdot {\vec{p}\:}^{\antiflip} + C \right) \\
& =
\vec{p} \cdot \phi \cdot \Delta(\cv + r \cdot \dv) \cdot \phi \cdot {\vec{p}\:}^{\antiflip}
+
\Delta(\vec{p}) \cdot \phi \cdot {\vec{p}\:}^{\antiflip}
+
\vec{p} \cdot \phi \cdot \Delta({\vec{p}\:}^{\antiflip})
+
\Delta(C) \\
& =
\Delta(\vec{p} \cdot \phi \cdot {\vec{p}\:}^{\antiflip} + C)
=
\Delta(\Psi_{M}) ,
\end{align*}
where we used equation~\eqref{equation_Delta_on_geometric_series} with $u = \cv + r \cdot \dv$.
Hence $\Psi_{M}$ satisfies the two equations of Theorem~\ref{theorem_unique_solution}.
\end{proof}

Note that we have been using indices for the matrix $M$ starting with $0$.
That is, entry $(i,j)$ satisfies $0 \leq i,j \leq 2r+1$.
\begin{corollary}
Let the index $i$ belong to the set $\{0\} \cup \{r+1,r+2, \ldots, 2r+1\}$
and $j$ to the set $\{0,1, \ldots, r\} \cup \{2r+1\}$.
Assume that $(i,j) \neq (0,2r+1)$.
Then the $\cd$-index of the interval 
$[(i,0), (j,k+1)]$ in the level poset associated to matrix $M$
is the polynomial
\begin{align*}
\sum_{w} r^{\text{number of $\dv$'s in $w$}} \cdot w
\end{align*}
where $w$ ranges over all $\cd$-monomials of degree $k$.
In the case $(i,j) = (0,2r+1)$, we need the extra assumption $k \geq 1$.
\label{corollary_M}
\end{corollary}
\begin{proof}
The result follows from the expansion~\eqref{equation_phi}.
\end{proof}

\section{The second class of level Eulerian posets}
\label{section_The_second_class}

We now introduce the second class of level Eulerian posets.
Note that the matrix $M$ of the previous section
and the matrix $N$ in this section are very close,
that is, they differ in four entries only.
However
the calculation of the $\cd$-series for the level poset associated to the matrix $N$
is much longer.

For this class we restrict the parameter $r$ to be greater than or equal to $2$.
We begin by introducing two new blocks:
\begin{itemize}
\item[$(A')$]
Let $A' = (a'_{i,j})_{0 \leq i,j \leq r}$
where $a'_{i,j} = 1$ if $i=0$ and $j \leq r-1$,
or $i=1$ and $j=r$. Otherwise let
$a'_{i,j} = 0$.
\item[$(F')$]
Let $F' = (f'_{i,j})_{0 \leq i,j \leq r}$
where
$f'_{i,j} = 1$ if
$i=0$ and $j \leq r-1$,
or $i \geq 1$ and $j=r$.
Otherwise let $f'_{i,j}=0$.
Note that ${F'}^{\antiflip} = F'$.
\end{itemize}

Similarly to Lemma~\ref{lemma_relations_M} we have:
\begin{lemma}
The following five relations hold among the
matrices $A$, $A'$, $B$, $F'$ and $J$:
\begin{align}
A'^{2} & = A,
&
{{A'}^{\antiflip}}^{2} & = A^{\antiflip}, 
&
J A' & = J = {A'}^{\antiflip} J . \\
\label{equation_N_F'}
&& A' \cdot B + B \cdot {A'}^{\antiflip} & = 2 \cdot F' .
\end{align}
\label{lemma_relations_N}
\end{lemma}
\begin{proof}
The first four relations are straightforward to verify.
Next $A' \cdot B$ has 
the first row to be $(2, \ldots, 2,1,0)$,
the second row to be $(0, \ldots, 0,1,1)$
and the remaining rows to be $0$.
Similarly,
$B \cdot {A'}^{\antiflip}$ has 
the last column to be $(0,1,2, \ldots, 2)$,
the next to last column to be $(1,1,0, \ldots, 0)$
and the previous columns to be $0$.
Adding these two matrix products yields relation~\eqref{equation_N_F'}.
\end{proof}
Let $N = N^{\antiflip}$ be the block matrix
\begin{align*}
N
=
\BMTwoByTwo{A'}{B}{J}{{A'}^{\antiflip}} .
\end{align*}
For $r = 3$ the matrices $N$ and $\Bin(N^{2})$ are as follows
\begin{align*}
N
& =
\left(\begin{array}{c c c c c c c c c}
1 & 1 & 1 & 0 & \VerticalLine & 1 & 0 & 0 & 0 \\
0 & 0 & 0 & 1 & \VerticalLine & 1 & 1 & 0 & 0 \\
0 & 0 & 0 & 0 & \VerticalLine & 0 & 1 & 1 & 0 \\
0 & 0 & 0 & 0 & \VerticalLine & 0 & 0 & 1 & 1 \\ \hline
1 & 1 & 1 & 1 & \VerticalLine & 0 & 0 & 1 & 0 \\
1 & 1 & 1 & 1 & \VerticalLine & 0 & 0 & 0 & 1 \\
1 & 1 & 1 & 1 & \VerticalLine & 0 & 0 & 0 & 1 \\
1 & 1 & 1 & 1 & \VerticalLine & 0 & 0 & 0 & 1
\end{array}\right) ,
&
\Bin(N^{2})
& =
\left(\begin{array}{c c c c c c c c c}
1 & 1 & 1 & 1 & \VerticalLine & 1 & 1 & 1 & 0 \\
1 & 1 & 1 & 1 & \VerticalLine & 0 & 0 & 1 & 1 \\
1 & 1 & 1 & 1 & \VerticalLine & 0 & 0 & 0 & 1 \\
1 & 1 & 1 & 1 & \VerticalLine & 0 & 0 & 0 & 1 \\ \hline
1 & 1 & 1 & 1 & \VerticalLine & 1 & 1 & 1 & 1 \\
1 & 1 & 1 & 1 & \VerticalLine & 1 & 1 & 1 & 1 \\
1 & 1 & 1 & 1 & \VerticalLine & 1 & 1 & 1 & 1 \\
1 & 1 & 1 & 1 & \VerticalLine & 1 & 1 & 1 & 1
\end{array}\right) .
\end{align*}
See Figure~\ref{figure_N_3} for the level poset associated this matrix.
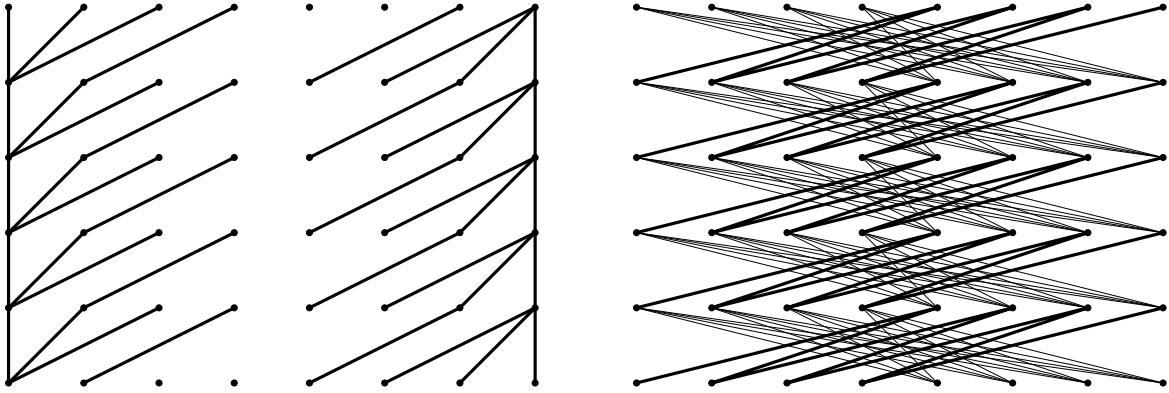
\begin{figure}
\begin{center}
\begin{tikzpicture}
\foreach \x in {0,1,2,3,4,5,6,7} {\foreach \y in {0,1,2,3,4,5}
{\draw[fill] (\x,\y) circle (0.04);};};

\foreach \y in {0,1,2,3,4}
{
\foreach \b in {0,1,2} {\draw[-,very thick] (0,{\y}) -- (\b,{\y+1});};
\draw[-,very thick] (1,{\y}) -- (3,{\y+1});
\foreach \a in {5,6,7} {\draw[-,very thick] (\a,{\y}) -- (7,{\y+1});};
\draw[-,very thick] (4,{\y}) -- (6,{\y+1});
};
\end{tikzpicture}
\hspace*{10mm}
\begin{tikzpicture}
\foreach \x in {0,1,2,3,4,5,6,7} {\foreach \y in {0,1,2,3,4,5}
{\draw[fill] (\x,\y) circle (0.04);};};

\foreach \y in {0,1,2,3,4}
{
\foreach \a in {4,5,6,7} {\foreach \b in {0,1,2,3} {\draw[-,very thin] (\a,{\y}) -- (\b,{\y+1});};};
\draw[-,very thick] (0,{\y}) -- (4,{\y+1}) -- (1,{\y}) -- (5,{\y+1}) -- (2,{\y})
                                         -- (6,{\y+1}) -- (3,{\y}) -- (7,{\y+1});
};
\end{tikzpicture}
\end{center}
\caption{First
the level poset associated to the northwest and southeast submatrices
of the matrix $N$.
Second,
the level poset associated to the northeast and southwest submatrices.
Again, the cover relations corresponding to the southwest block $J$ has been drawn
with thin lines.
By superimposing these two level posets
we obtain the level Eulerian posets associated to matrix $N$ for $r=3$.}
\label{figure_N_3}
\end{figure}

\begin{proposition}
For $r \geq 2$ the matrix $N$ is primitive with the exponent $3$ and encodes a level Eulerian poset.
\end{proposition}
\begin{proof}
This proof is similar to the proof of 
Proposition~\ref{proposition_M}
with a few changes.
First note that
\begin{align*}
N^{2} & = \BMTwoByTwo{2J}{2F'}{2J}{2J} = 2 \cdot \Bin(N^{2}),
&
\left(J - \Bin(N^{2})\right)^{2} & = \BMTwoByTwo{0}{J-F'}{0}{0}^{2} = 0 ,
\end{align*}
where the first identity follows
by the relations in Lemmas~\ref{lemma_relations_M} and~\ref{lemma_relations_N}.
Hence the rank $2$ Eulerian condition holds.
Next we have
$\Bin(N^{3}) = J$ so the exponent is $3$.
The row and column sums of
$- N + \Bin(N^{2})$
is
$r = n/2 - 1$ and hence
the Eulerian condition for even ranks greater than $4$ follows.
As noted above, $\left(J - \Bin(N^{2})\right)^{2}$
is the zero matrix and the rank $4$ Eulerian condition holds.
\end{proof}

We now turn to computing the $\cd$-index of the level poset for the matrix $N$.
Define the column vector $\vec{q}$ by the next equation
and we also display the row vector
${\vec{q}\:}^{\antiflip}$:
\begin{align*}
{\vec{q}\:}^{T}
& =
(\cv+(r-1)\dv, \cv+\dv, \underbrace{\cv, \ldots, \cv}_{r-1}, \underbrace{1, \ldots, 1}_{r+1}) , \\
{\vec{q}\:}^{\antiflip}
& =
(\underbrace{1, \ldots, 1}_{r+1}, \underbrace{\cv, \ldots, \cv}_{r-1}, \cv+\dv, \cv+(r-1)\dv) .
\end{align*}

Let $G$ and $X$ be the two following blocks:
\begin{itemize}
\item[$(G)$]
Let $G = (g_{i,j})_{0 \leq i,j \leq r}$
where
$g_{0,r} = r-2$,
$g_{0,r-1} = g_{1,r} = 1$
and
otherwise let $g_{i,j} = 0$.
Observe that $G^{\antiflip} = G$.
\item[$(X)$]
Let $X = (x_{i,j})_{0 \leq i,j \leq r}$
where
$x_{0,r} = r-1$,
$x_{0,r-1} = 1$
and
otherwise let $x_{i,j} = 0$.
\end{itemize}
\begin{lemma}
The following four relations hold among the five matrices $A$, $A'$, $F'$, $G$ and $X$:
\begin{align}
A' \cdot F' & = A' + X,
&
A' \cdot G & = X,
&
A' - A + X & = G = A'^{\antiflip}-A^{\antiflip}+X^{\antiflip}.
\end{align}
\label{lemma_relations_N_2}
\end{lemma}
This lemma is straightforward to verify and hence the proof is omitted.

Finally let $D = D^{\antiflip}$ be the block matrix
\begin{align*}
D
=
\BMTwoByTwo{A'}{B}{0}{A'^{\antiflip}}
+
\BMTwoByTwo{0}{F'}{0}{0} \cdot \cv
+
\BMTwoByTwo{0}{G}{0}{0} \cdot \dv .
\end{align*}
\begin{proposition}
For $r \geq 2$ the $\cd$-index series for the level Eulerian poset associated to the matrix~$N$ is given by
\begin{align*}
\Psi_{N} 
=
\vec{q} \cdot \phi \cdot {\vec{q}\:}^{\antiflip} + D .
\end{align*}
\end{proposition}
\begin{proof}
Again the first condition of Theorem~\ref{theorem_unique_solution}
is straightforward to check:
\begin{align*}
\Psi_{N}\vrule_{\: \cv=t,\dv=0}
& =
{\vec{q}\:\vrule_{\: \cv=t, \dv=0}} \cdot \frac{1}{1-t} \cdot {{\vec{q}\:}^{\antiflip}\:\vrule_{\: \cv=t, \dv=0}}
+
{D\:\vrule_{\: \cv=t, \dv=0}} \\
& =
\BMTwoByTwo{tJ}{t^{2}J}{J}{tJ} \cdot \frac{1}{1-t} 
+
\BMTwoByTwo{A'}{B + tF'}{0}{A'^{\antiflip}}
=
K_{N}(t) .
\end{align*}
Next we develop identities that we use in showing the second condition.
Note that
\begin{align*}
{\vec{q}\:}^{\antiflip} \cdot \tv \cdot \vec{q}
& =
r\cv\tv 
+ 
\tv r\cv
+ 
(\cv+r\dv) \tv 
+ 
\tv (\cv+r\dv)  .
\end{align*}
Now using that $\phi \cdot (\cv+r\dv) = \phi - 1 = (\cv+r\dv) \cdot \phi$
and $\Delta(\phi) = \phi \cdot \Delta(\cv + r \dv) \cdot \phi$
we have
\begin{align*}
\phi \cdot {\vec{q}\:}^{\antiflip} \cdot \tv \cdot \vec{q} \cdot \phi
& =
\phi \cdot (r\cv\tv + \tv r\cv) \cdot \phi
+ 
\phi \cdot (\cv+r\dv) \cdot \tv \cdot \phi
+ 
\phi \cdot \tv \cdot (\cv+r\dv) \cdot \phi \\
& =
\phi \cdot \Delta(r\dv) \cdot \phi
+ 
(\phi-1) \cdot \tv \cdot \phi
+ 
\phi \cdot \tv \cdot (\phi-1) \\
& =
\phi \cdot \Delta(\cv+r\dv) \cdot \phi
- 
\tv \cdot \phi
- 
\phi \cdot \tv \\
& =
\Delta(\phi)
- 
\tv \cdot \phi
- 
\phi \cdot \tv ,
\end{align*}
where the last step is equation~\eqref{equation_Delta_on_geometric_series}.
Left multiply with $\vec{q}$ and right multiply with~${\vec{q}\:}^{\antiflip}$ to obtain
\begin{align}
\label{equation_N_1}
\vec{q} \cdot \phi \cdot {\vec{q}\:}^{\antiflip} \cdot \tv \cdot \vec{q} \cdot \phi \cdot {\vec{q}\:}^{\antiflip}
& =
\vec{q} \cdot \Delta(\phi) \cdot {\vec{q}\:}^{\antiflip}
- 
\vec{q} \cdot \tv \cdot \phi \cdot {\vec{q}\:}^{\antiflip}
- 
\vec{q} \cdot \phi \cdot \tv \cdot {\vec{q}\:}^{\antiflip} .
\end{align}
Next we calculate
\begin{align*}
&
{\vec{q}\:}^{\antiflip} \cdot \tv \cdot D
-
\Delta({\vec{q}\:}^{\antiflip})
-
\tv \cdot {\vec{q}\:}^{\antiflip} \\
& =
{\vec{q}\:}^{\antiflip} \cdot \BMTwoByTwo{A'}{B}{0}{A'^{\antiflip}} \cdot  \tv
+
{\vec{q}\:}^{\antiflip} \cdot \BMTwoByTwo{0}{F'}{0}{0} \cdot  \tv \cv
+
{\vec{q}\:}^{\antiflip} \cdot \BMTwoByTwo{0}{G}{0}{0} \cdot \tv \dv
-
\Delta({\vec{q}\:}^{\antiflip})
-
\tv \cdot {\vec{q}\:}^{\antiflip} \\
& =
(\underbrace{1, \ldots, 1}_{r+1}, \underbrace{2, \ldots, 2}_{r-1}, 2+\cv, 1+r\cv+r\dv) \cdot \tv \\
& +
(\underbrace{0, \ldots, 0}_{r+1}, \underbrace{1, \ldots, 1}_{r-1}, 2, r) \cdot \tv\cv \\
& +
(\underbrace{0, \ldots, 0}_{r+1}, \underbrace{0, \ldots, 0}_{r-1}, 1, r-1) \cdot \tv\dv \\
& -
(\underbrace{0, \ldots, 0}_{r+1}, \underbrace{2\tv, \ldots, 2\tv}_{r-1},
2\tv+\cv\tv+\tv\cv, 2\tv+(r-1)\cv\tv+(r-1)\tv\cv) \\
& -
(\underbrace{\tv, \ldots, \tv}_{r+1}, \underbrace{\tv\cv, \ldots, \tv\cv}_{r-1},
\tv\cv+\tv\dv, \tv\cv+(r-1)\tv\dv) \\
& =
(\underbrace{0, \ldots, 0}_{r+1}, \underbrace{0, \ldots, 0}_{r-1},
0, -\tv + \cv\tv+r\dv\tv) .
\end{align*}
Move the two terms $\Delta({\vec{q}\:}^{\antiflip})$
and $\tv \cdot {\vec{q}\:}^{\antiflip}$
to the left-hand side and left multiply with $\vec{q} \cdot \phi$ yields
\begin{align}
\label{equation_N_2}
\vec{q} \cdot \phi \cdot {\vec{q}\:}^{\antiflip} \cdot \tv \cdot D
& =
\vec{q} \cdot \phi \cdot \Delta({\vec{q}\:}^{\antiflip})
+
\vec{q} \cdot \phi \cdot \tv \cdot {\vec{q}\:}^{\antiflip}
-
\vec{q} \cdot \phi \cdot 
(1 - \cv - r \dv) \cdot \tv \cdot 
(\underbrace{0, \ldots, 0}_{2r+1}, 1) \\
& =
\vec{q} \cdot \phi \cdot \Delta({\vec{q}\:}^{\antiflip})
+
\vec{q} \cdot \phi \cdot \tv \cdot {\vec{q}\:}^{\antiflip}
-
\vec{q} \cdot \tv \cdot 
(\underbrace{0, \ldots, 0}_{2r+1}, 1) 
\nonumber \\
& =
\vec{q} \cdot \phi \cdot \Delta({\vec{q}\:}^{\antiflip})
+
\vec{q} \cdot \phi \cdot \tv \cdot {\vec{q}\:}^{\antiflip}
-
\BMTwoByTwo{0}{A^{\antiflip} \cdot \cv\tv + X^{\antiflip} \cdot \dv\tv}{0}{A^{\antiflip} \cdot \tv} ,
\nonumber
\end{align}
where we used $\phi \cdot (1 - \cv - r \dv) = 1$.
Applying the anti-diagonal flip, that is, $M \longmapsto M^{\antiflip}$, yields
\begin{align}
\label{equation_N_3}
D \cdot \tv \cdot \vec{q} \cdot \phi \cdot {\vec{q}\:}^{\antiflip}
& =
\Delta(\vec{q}) \cdot \phi \cdot {\vec{q}\:}^{\antiflip}
+
\vec{q} \cdot \tv \cdot \phi \cdot {\vec{q}\:}^{\antiflip}
-
\BMTwoByTwo{A \cdot \tv}{A \cdot \tv\cv + X \cdot \tv\dv}{0}{0} .
\end{align}
Finally we expand $D \cdot \tv \cdot D$:
\begin{align}
\label{equation_N_4}
D \cdot \tv \cdot D
& =
\BMTwoByTwo{A'}{B}{0}{A'^{\antiflip}}^{2} \cdot \tv
+
\BMTwoByTwo{A'}{B}{0}{A'^{\antiflip}}
\cdot
\BMTwoByTwo{0}{F'}{0}{0} \cdot \tv\cv
+
\BMTwoByTwo{A'}{B}{0}{A'^{\antiflip}}
\cdot
\BMTwoByTwo{0}{G}{0}{0} \cdot \tv \dv \\
&
+
\BMTwoByTwo{0}{F'}{0}{0}
\cdot
\BMTwoByTwo{A'}{B}{0}{A'^{\antiflip}} \cdot \cv\tv
+
\BMTwoByTwo{0}{G}{0}{0}
\cdot
\BMTwoByTwo{A'}{B}{0}{A'^{\antiflip}} \cdot \dv \tv 
\nonumber \\
& =
\BMTwoByTwo{A}{2F'}{0}{A^{\antiflip}} \cdot \tv
+
\BMTwoByTwo{0}{A'+X}{0}{0} \cdot \tv\cv
+
\BMTwoByTwo{0}{X}{0}{0} \cdot \tv \dv 
\nonumber \\
&
+
\BMTwoByTwo{0}{A'^{\antiflip}+X^{\antiflip}}{0}{0} \cdot \cv\tv
+
\BMTwoByTwo{0}{X^{\antiflip}}{0}{0} \cdot \dv \tv ,
\nonumber
\end{align}
where relations from Lemmas~\ref{lemma_relations_N} 
and~\ref{lemma_relations_N_2}
were used.
Finally adding the four identities~\eqref{equation_N_1}
through~\eqref{equation_N_4} yields
\begin{align*}
\Psi_{N} \cdot \tv \cdot \Psi_{N}
& =
\left( \vec{q} \cdot \phi \cdot {\vec{q}\:}^{\antiflip} + D \right)
\cdot \tv \cdot
\left( \vec{q} \cdot \phi \cdot {\vec{q}\:}^{\antiflip} + D \right) \\
& =
\vec{q} \cdot \phi \cdot {\vec{q}\:}^{\antiflip}
\cdot \tv \cdot
\vec{q} \cdot \phi \cdot {\vec{q}\:}^{\antiflip}
+
D
\cdot \tv \cdot
\vec{q} \cdot \phi \cdot {\vec{q}\:}^{\antiflip}
+
\vec{q} \cdot \phi \cdot {\vec{q}\:}^{\antiflip}
\cdot \tv \cdot
D
+
D
\cdot \tv \cdot
D \\
& =
\vec{q} \cdot \Delta(\phi) \cdot {\vec{q}\:}^{\antiflip}
+
\vec{q} \cdot \phi \cdot \Delta({\vec{q}\:}^{\antiflip})
+
\Delta(\vec{q}) \cdot \phi \cdot {\vec{q}\:}^{\antiflip} \\
&
+
\BMTwoByTwo{0}{F'}{0}{0} \cdot 2\tv
+
\BMTwoByTwo{0}{A'-A+X}{0}{0} \cdot \tv\cv
+
\BMTwoByTwo{0}{A'^{\antiflip}-A^{\antiflip}+X^{\antiflip}}{0}{0} \cdot \cv\tv \\
& =
\Delta(\vec{q} \cdot \phi \cdot {\vec{q}\:}^{\antiflip})
+
\BMTwoByTwo{0}{F'}{0}{0} \cdot \Delta(\cv)
+
\BMTwoByTwo{0}{G}{0}{0} \cdot \Delta(\dv) \\
& =
\Delta(\vec{q} \cdot \phi \cdot {\vec{q}\:}^{\antiflip} + D) = \Delta(\Psi_{N}) .
\end{align*}
Hence $\Psi_{N}$ satisfies both equations of Theorem~\ref{theorem_unique_solution}.
\end{proof}

Similarly to Corollary~\ref{corollary_M} we have the next statement.
\begin{corollary}
Let $i$ belong to the interval $\{r+1,r+2, \ldots, 2r+1\}$
and $j$ to the interval $\{0,1, \ldots, r\}$.
For $r \geq 2$ the $\cd$-index of the interval 
$[(i,0), (j,k+1)]$ in the level poset associated to matrix $N$
is the sum
\begin{align*}
\sum_{w} r^{\text{number of $\dv$'s in $w$}} \cdot w
\end{align*}
where $w$ ranges over all $\cd$-monomials of degree $k$.
\label{corollary_N}
\end{corollary}

\section{Sphericity of the first class}

We now state Theorem~6.7 of \cite{Ehrenborg_Hetyei_Readdy}
that gives an algebraic way to verify that level posets are shellable.
We begin by defining the associated {\em algebra of walks}.
For a $0,1$-matrix $M = (m_{i,j})_{1 \leq i,j \leq n}$
the algebra of walks $\Qqq\langle M \rangle$
is the quotient of the free non-commutative algebra over $\Qqq$
generated by the set of variables
$\{x_{i,j} : m_{i,j} = 1\}$
quoted out by the ideal generated
by the elements $x_{i,j} \cdot x_{i',j'}$ where $j \neq i'$.
That is, a linear basis for this algebra is naturally indexed by
all the walks in the digraph associated to the matrix $M$.
In fact, we introduce the shorthand notation
$x_{i_{1},i_{2}, \ldots, i_{k}}$ to denoted the product
$x_{i_{1},i_{2}} \cdot x_{i_{2},i_{3}} \cdots x_{i_{k-1},i_{k}}$.

Let $I_{<}$ be the ideal in $\Qqq\langle M \rangle$
generated by all the monomials
$x_{i,j,k} = x_{i,j} \cdot x_{j,k}$
such that
there exists $j' < j$ such that $m_{i,j'} = m_{j',k} = 1$.
Observe that now we use the natural order on the indices $1,2, \ldots, n$ of matrix $M$.
Let $Z_{M} = (z_{i,j})_{1 \leq i,j \leq n}$ be the matrix where
$z_{i,j} = x_{i,j}$ if $m_{i,j} = 1$ and $z_{i,j} = 0$ otherwise.
Our interest is to compute powers of this matrix in the quotient ring
$\Qqq\langle M \rangle/I_{<}$.

Another way to describe the monomials occurring in the ideal $I_{<}$ is that
it consists of all degree $2$ monomials that do not
occur in the matrix $Z_{M}^{2}$
when calculated in $\Qqq\langle M \rangle/I_{<}$.
This fact will be useful in explicit calculations.

Now Theorem~6.7 in~\cite{Ehrenborg_Hetyei_Readdy} states:
\begin{theorem}[Ehrenborg--Hetyei--Readdy]
If over the ring $\Qqq\langle M \rangle/I_{<}$ every entry in every power
of the matrix $Z_{M}$ is zero or a single monomial
then $<$ is a vertex shelling order, implying that
every interval in the level poset is shellable.
\label{theorem_shellable}
\end{theorem}

Theorem~6.9
in~\cite{Ehrenborg_Hetyei_Readdy}
states that if an Eulerian poset is shellable
then its order complex is homeomorphic to a sphere.
Hence we have the conclusion:
\begin{corollary}[Ehrenborg--Hetyei--Readdy]
If each interval in a level Eulerian poset 
is shellable then the order complex of every interval is homeomorphic to a sphere.
\end{corollary}

We now use Theorem~\ref{theorem_shellable} to show that all the intervals
in the class of level Eulerian posets from
Section~\ref{section_The_first_class} are shellable.
We begin to show the matrices $Z_{M}$ and $Z_{M}^{2}$
when $r=3$ for this class:
\begin{align*}
Z_{M}
& =
\begin{pmatrix}
x_{0,0} & x_{0,1} & x_{0,2} & x_{0,3} & \VerticalLine & x_{0,4} & 0 & 0 & 0 \\
0 & 0 & 0 & 0 & \VerticalLine & x_{1,4} & x_{1,5} & 0 & 0 \\
0 & 0 & 0 & 0 & \VerticalLine & 0 & x_{2,5} & x_{2,6} & 0 \\
0 & 0 & 0 & 0 & \VerticalLine & 0 & 0 & x_{3,6} & x_{3,7} \\ \hline
x_{4,0} & x_{4,1} & x_{4,2} & x_{4,3} & \VerticalLine & 0 & 0 & 0 & x_{4,7} \\
x_{5,0} & x_{5,1} & x_{5,2} & x_{5,3} & \VerticalLine & 0 & 0 & 0 & x_{5,7} \\
x_{6,0} & x_{6,1} & x_{6,2} & x_{6,3} & \VerticalLine & 0 & 0 & 0 & x_{6,7} \\
x_{7,0} & x_{7,1} & x_{7,2} & x_{7,3} & \VerticalLine & 0 & 0 & 0 & x_{7,7}
\end{pmatrix}
 \\
Z_{M}^{2}
& =
\begin{pmatrix}
x_{0,0,0} & x_{0,0,1} & x_{0,0,2} & x_{0,0,3} & \VerticalLine & x_{0,0,4} & x_{0,1,5} & x_{0,2,6} & x_{0,3,7} \\
x_{1,4,0} & x_{1,4,1} & x_{1,4,2} & x_{1,4,3} & \VerticalLine & 0 & 0 & 0 & x_{1,4,7} \\
x_{2,5,0} & x_{2,5,1} & x_{2,5,2} & x_{2,5,3} & \VerticalLine & 0 & 0 & 0 & x_{2,5,7} \\
x_{3,6,0} & x_{3,6,1} & x_{3,6,2} & x_{3,6,3} & \VerticalLine & 0 & 0 & 0 & x_{3,6,7} \\ \hline
x_{4,0,0} & x_{4,0,1} & x_{4,0,2} & x_{4,0,3} & \VerticalLine & x_{4,0,4} & x_{4,1,5} & x_{4,2,6} & x_{4,3,7} \\
x_{5,0,0} & x_{5,0,1} & x_{5,0,2} & x_{5,0,3} & \VerticalLine & x_{5,0,4} & x_{5,1,5} & x_{5,2,6} & x_{5,3,7} \\
x_{6,0,0} & x_{6,0,1} & x_{6,0,2} & x_{6,0,3} & \VerticalLine & x_{6,0,4} & x_{6,1,5} & x_{6,2,6} & x_{6,3,7} \\
x_{7,0,0} & x_{7,0,1} & x_{7,0,2} & x_{7,0,3} & \VerticalLine & x_{7,0,4} & x_{7,1,5} & x_{7,2,6} & x_{7,3,7}
\end{pmatrix}
\end{align*}
The next proposition is the full evaluation of the matrix $Z_{M}^{p}$.
We use the notation
$0^{p}$ to denote a string of $p$ zeroes, that is, 
$0^{p} = \underbrace{0,0, \ldots,0}_{p}$.
\begin{proposition}
Let $M$ be the level Eulerian poset from
Section~\ref{section_The_first_class}
and let the power $p$ be at least $2$. 
The entries of the southwest, southeast and northwest blocks of the matrix $Z_{M}^{p}$
evaluated in the ring $\Qqq\langle M \rangle/I_{<}$
are given by
\begin{align*}
(Z_{M}^{p})_{i,j} 
& =
\begin{cases}
x_{i,0^{p-1},j}
& \text{ if } r+1 \leq i \leq 2r+1, 0 \leq j \leq r, \\
x_{i,0^{p-2},j-r-1,j}
& \text{ if } r+1 \leq i,j \leq 2r+1, \\
x_{0^{p},j}
& \text{ if } i=0, 0 \leq j \leq r, \\
x_{i,i+r+1,0^{p-2},j}
& \text{ if } i \neq 0, 0 \leq i,j \leq r. \\
\end{cases}
\end{align*}
For the northeast block, that is, $0 \leq i \leq r$ and $r+1 \leq j \leq 2r+1$,
 the entries are described by
\begin{align*}
(Z_{M}^{p})_{i,j} 
& =
\begin{cases}
x_{0^{p-1},j-r-1,j}
& \text{ if } i=0, r+1 \leq j \leq 2r+1, \\
0
& \text{ if } p=2, 1 \leq i \leq r, r+1 \leq j \leq 2r, \\
x_{i,i+r+1,2r+1}
& \text{ if } p=2, 1 \leq i \leq r, j = 2r+1, \\
x_{i, i+r+1,0^{p-3},j-r-1,j}
& \text{ if } p \geq 3, 1 \leq i \leq r, r+1 \leq j \leq 2r+1. \\
\end{cases}
\end{align*}
\label{proposition_M_is_shellable}
\end{proposition}
\begin{proof}
The proof is by induction on the power $p$.
The induction basis $p=2$ is straightforward.
Assume now that the result holds for $p \geq 2$
and we prove it for $p+1$.
There are six cases to consider.
Note that in each calculation the two first steps takes place
in $\Qqq\langle M \rangle$ and the third and final step
in the quotient $\Qqq\langle M \rangle/I_{<}$.
\begin{enumerate}
\item[(i)] $r+1 \leq i \leq 2r+1, 0 \leq j \leq r$:
\begin{align*}
(Z_{M}^{p+1})_{i,j}
& =
\sum_{\substack{\ell = 0 \text{ or} \\ r+2 \leq \ell \leq 2r+1}}
(Z_{M}^{p})_{i,\ell} \cdot (Z_{M})_{\ell,j} \\
& =
x_{i,0^{p-1},0} \cdot x_{0,j}
+
\sum_{\ell = r+2}^{2r+1}
x_{i,0^{p-1},\ell} \cdot x_{\ell,j}
=
x_{i,0^{p},j} .
\end{align*}
\item[(ii)] $r+1 \leq i,j \leq 2r+1$:
First the subcase where $r+1 \leq j \leq 2r$.
\begin{align*}
(Z_{M}^{p+1})_{i,j}
& =
\sum_{\ell = j-r-1}^{j-r}
(Z_{M}^{p})_{i,\ell} \cdot (Z_{M})_{\ell,j} \\
& =
\sum_{\ell = j-r-1}^{j-r}
x_{i,0^{p-1},\ell} \cdot x_{\ell,j}
=
x_{i,0^{p-1},j-r-1,j} .
\end{align*}
Second, the subcase $j=2r+1$.
\begin{align*}
(Z_{M}^{p+1})_{i,2r+1}
& =
\sum_{\ell = r}^{2r+1}
(Z_{M}^{p})_{i,\ell} \cdot (Z_{M})_{\ell,2r+1} \\
& =
x_{i,0^{p-1},r} \cdot x_{r,2r+1}
+
\sum_{\ell = r+1}^{2r+1}
x_{i,0^{p-2},\ell-r-1,\ell} \cdot x_{\ell,2r+1}
=
x_{i,0^{p-1},r,2r+1} .
\end{align*}
\item[(iii)] $i=0, 0 \leq j \leq r$:
\begin{align*}
(Z_{M}^{p+1})_{0,j}
& =
\sum_{\substack{\ell = 0 \text{ or} \\ r+1 \leq \ell \leq 2r+1}}
(Z_{M}^{p})_{0,\ell} \cdot (Z_{M})_{\ell,j} \\
& =
x_{0^{p+1}} \cdot x_{0,j} 
+
\sum_{\ell = r+1}^{2r+1}
x_{0^{p-1},\ell-r-1,\ell} \cdot x_{\ell,j}
=
x_{0^{p+1},j} .
\end{align*}
\item[(iv)] $i \neq 0, 0 \leq i,j \leq r$:
First consider $p=2$:
\begin{align*}
(Z_{M}^{3})_{i,j}
& =
(Z_{M}^{2})_{i,0} \cdot (Z_{M})_{0,j}  + (Z_{M}^{2})_{i,2r+1} \cdot (Z_{M})_{2r+1,j} \\
& =
x_{i,i+r+1,0} \cdot x_{0,j}  + x_{i,i+r+1,2r+1} \cdot x_{2r+1,j}
=
x_{i,i+r+1,0,j} .
\end{align*}
Second consider the subcase $p \geq 3$:
\begin{align*}
(Z_{M}^{p+1})_{i,j}
& =
(Z_{M}^{p})_{i,0} \cdot (Z_{M})_{0,j}  + (Z_{M}^{p})_{i,2r+1} \cdot (Z_{M})_{2r+1,j} \\
& =
x_{i,i+r+1,0^{p-2},0} \cdot x_{0,j}  + x_{i,i+r+1,0^{p-3},r,2r+1} \cdot x_{2r+1,j}
=
x_{i,i+r+1,0^{p-1},j} .
\end{align*}
\item[(v)] $i=0, r+1 \leq j \leq 2r+1$:
First subcase $r+1 \leq j \leq 2r$, that is, $j \neq 2r+1$:
\begin{align*}
(Z_{M}^{p+1})_{0,j}
& =
(Z_{M}^{p})_{0,j-r-1} \cdot (Z_{M})_{j-r-1,j} + (Z_{M}^{p})_{0,j-r} \cdot (Z_{M})_{j-r,j} \\
& =
x_{0^{p},j-r-1} \cdot x_{j-r-1,j} + x_{0^{p},j-r} \cdot x_{j-r,j}
=
x_{0^{p},j-r-1,j} .
\end{align*}
Second subcase, $j = 2r+1$:
\begin{align*}
(Z_{M}^{p+1})_{0,2r+1}
& =
(Z_{M}^{p})_{0,r} \cdot (Z_{M})_{r,2r+1}
+
\sum_{\ell = r+1}^{2r+1} (Z_{M}^{p})_{0,\ell} \cdot (Z_{M})_{\ell,2r+1} \\
& =
x_{0^{p},r} \cdot x_{r,2r+1}
+
\sum_{\ell = r+1}^{2r+1} x_{0^{p-1},\ell-r-1,\ell} \cdot x_{\ell,2r+1}
=
x_{0^{p},r,2r+1} .
\end{align*}
\item[(vi)] $1 \leq i \leq r, r+1 \leq j \leq 2r+1$:
First subcase $r+1 \leq j \leq 2r$:
\begin{align*}
(Z_{M}^{p+1})_{i,j}
& =
(Z_{M}^{p})_{i,j-r-1} \cdot (Z_{M})_{j-r-1,j}
+
(Z_{M}^{p})_{i,j-r} \cdot (Z_{M})_{j-r,j} \\
& =
x_{i,i+r+1,0^{p-2},j-r-1} \cdot x_{j-r-1,j}
+
x_{i,i+r+1,0^{p-2},j-r} \cdot x_{j-r,j} \\
& =
x_{i,i+r+1,0^{p-2},j-r-1,j} .
\end{align*}
Note that for the last conclusion,
one has to consider the two cases $p=2$ and $p \geq 3$ separately. 
The second subcase $j = 2r+1$ and $p=2$:
\begin{align*}
(Z_{M}^{3})_{i,2r+1}
& =
(Z_{M}^{2})_{i,r} \cdot (Z_{M})_{r,2r+1}
+
(Z_{M}^{2})_{i,2r+1} \cdot (Z_{M})_{2r+1,2r+1} \\
& =
x_{i,i+r+1,r} \cdot x_{r,2r+1}
+
x_{i,i+r+1,2r+1} \cdot x_{2r+1,2r+1} \\
& =
x_{i,i+r+1,r,2r+1} .
\end{align*}
Lastly, the third subcase $j = 2r+1$ and $p \geq 3$:
\begin{align*}
(Z_{M}^{p+1})_{i,2r+1}
& =
(Z_{M}^{p})_{i,r} \cdot (Z_{M})_{r,2r+1}
+
\sum_{\ell = r+1}^{2r+1}
(Z_{M}^{p})_{i,\ell} \cdot (Z_{M})_{\ell,2r+1} \\
& =
x_{i,i+r+1,0^{p-2},r} \cdot x_{r,2r+1}
+
\sum_{\ell = r+1}^{2r+1}
x_{i,i+r+1,0^{p-3},\ell-r-1,\ell} \cdot x_{\ell,2r+1} \\
& =
x_{i,i+r+1,0^{p-2},r,2r+1} .
\end{align*}
\end{enumerate}
This argument completes the six cases and their subcases.
\end{proof}

\begin{corollary}
The order complexes of the intervals of the level Eulerian poset in
Section~\ref{section_The_first_class} are homeomorphic to spheres.
\end{corollary}

We turn now to the second class of level Eulerian posets,
that is, the class from Section~\ref{section_The_second_class}.
However, in this case this method for proving shellability fails.
We show this negative result by computing the entry
$(Z_{N}^{3})_{1,2r+1}$.
Note that in the associated digraph the node $1$ can only go to the three nodes
$r$, $r+1$ and $r+2$.
Begin to observe
\begin{align*}
(Z_{N}^{2})_{r,2r+1}
& =
x_{r,2r} \cdot x_{2r,2r+1} + x_{r,2r+1} \cdot x_{2r+1,2r+1}
=
x_{r,2r} \cdot x_{2r,2r+1} 
=
x_{r,2r,2r+1} , \\
(Z_{N}^{2})_{r+1,2r+1}
& =
x_{r+1,r} \cdot x_{r,2r+1} + x_{r+1,2r} \cdot x_{2r,2r+1}
=
x_{r+1,r} \cdot x_{r,2r+1}
=
x_{r+1,r,2r+1} , \\
(Z_{N}^{2})_{r+2,2r+1}
& =
x_{r+2,r} \cdot x_{r,2r+1} + x_{r+2,2r+1} \cdot x_{2r+1,2r+1}
=
x_{r+2,r} \cdot x_{r,2r+1}
=
x_{r+2,r,2r+1} .
\end{align*}
Finally, we calculate
\begin{align*}
(Z_{N}^{3})_{1,2r+1}
& =
x_{1,r} \cdot (Z_{N}^{2})_{r,2r+1}
+
x_{1,r+1} \cdot (Z_{N}^{2})_{r+1,2r+1}
+
x_{1,r+2} \cdot (Z_{N}^{2})_{r+2,2r+1} \\
& =
x_{1,r} \cdot x_{r,2r,2r+1}
+
x_{1,r+1} \cdot x_{r+1,r,2r+1}
+
x_{1,r+2} \cdot x_{r+2,r,2r+1} \\
& =
x_{1,r,2r,2r+1}
+
(x_{1,r+1} \cdot x_{r+1,r} + x_{1,r+2} \cdot x_{r+2,r}) \cdot x_{r,2r+1} \\
& =
x_{1,r,2r,2r+1}
+
x_{1,r+1,r,2r+1} .
\end{align*}
Note that this last expression do not reduce any further.
However, we end this section with the natural conjecture.
\begin{conjecture}
The intervals in the level Eulerian poset from
Section~\ref{section_The_second_class}
are shellable and hence their order complexes
are homeomorphic to spheres.
\end{conjecture}

\section{Concluding remarks}

Are there better ways to compute the $\cd$-index series of level Eulerian posets?
Even though the two classes $M$ and $N$ in this paper are very close, that is,
the Hamming distance between these two matrices is $4$, the $\cd$-index series
verification for the matrix $N$ is a lot longer.

Also note that the classes $M$ and $N$ share many properties, such as
they both have the exponent to be $3$. Also the denominator in their
$\cd$-index series is the polynomial expression $1 - \cv - r \cdot \dv$.
Does the exponent $\gamma$ yield a bound on the degree of the
denominator of the rational expression of the $\cd$-index series?

Instead of the finding a rational expression for the $\cd$-index series $\Psi$ for a
level Eulerian poset, would it be more compact to express the series as a weighted finite
automata? See the book~\cite[Chapter~1]{Berstel_Reutenauer} for more on
the connection between rational series and weighted finite automata.
Given a $0,1$-matrix $M$ indexed by $0 \leq i,j \leq n-1$.
Consider the set of states
$\{s_{0}, s_{1}, \ldots, s_{n-1}, f_{0}, f_{1}, \ldots, f_{n-1}\}$.
We have two types of directed edges
$s_{i} \longrightarrow f_{j}$ and
$f_{i} \longrightarrow f_{j}$
and we weight them according to
\begin{align*}
w(s_{i} \longrightarrow f_{j})
& =
(K_{M}(\av-\bv))_{i,j} ,
&
w(f_{i} \longrightarrow f_{j})
& =
\bv \cdot (K_{M}(\av-\bv))_{i,j} ,
\end{align*}
where $(K_{M}(\av-\bv))_{i,j}$ is the $(i,j)$ entry of
the series $K_{M}$ from equation~\eqref{equation_K}.
Then the $\ab$-series $\Psi_{i,j}$ is the sum of the total weight
of all the paths from the starting state $s_{i}$ to the final state $f_{j}$.
However, even for the first non-trivial example, that is,
when $r=1$ in the first class of level Eulerian posets,
it is a non-trivial task to see that the associated weighted automata
generate a $\cd$-series.

Corollaries~\ref{corollary_M} and~\ref{corollary_N} raises another question:
For which functions $f$ does there exist an Eulerian poset of rank $k+1$ such that
the $\cd$-index of the poset is
\begin{align*}
\sum_{w} f(\text{number of $\dv$'s in $w$}) \cdot w 
\end{align*}
where the sum ranges over all $\cd$-monomials of degree $k$?

There is another class of level Eulerian posets that
we have not touch upon, namely half Eulerian posets.
Half Eulerian posets were developed by Bayer and Hetyei
in the two papers~\cite{Bayer_Hetyei_1,Bayer_Hetyei_2}.
They are so named that after a doubling operation they yield
Eulerian posets. See Example~6.11
in~\cite{Ehrenborg_Hetyei_Readdy} for a level poset that
is half Eulerian.

\section*{Acknowledgments}

The author thanks Margaret Readdy and the two referees for their comments
on an earlier draft of this paper.
This work was partially supported by a grant from the
Simons Foundation
(\#854548 to Richard~Ehrenborg).

\bibliographystyle{plain}
\bibliography{poset.bib}

\end{document}